\documentclass[12pt,leqno]{article}
\usepackage[a4paper,left=35mm,right=35mm,top=55mm,bottom=55mm,marginpar=20mm]{geometry}
\usepackage{fancyhdr}
\usepackage{amsmath,amssymb}
\usepackage{fouriernc}

\let\nulset\emptyset        
\renewcommand{\emptyset}{\text{\large$\nulset$}}

\def\R{\mathbb{R}}\def\C{\mathbb{C}}
\def\S{$\mathsection$}
\renewcommand{\epsilon}{\varepsilon}

\newcounter{sequation}[section] 
\newcounter{pequation} 
\newenvironment{proof}{\setcounter{pequation}{0}\noindent}{\,\,\nobreak$\square$\par} 

\title{\vskip-.6in         New Korn and projection estimates  on domains for a class of constant-rank operators \author{\scshape Adolfo Arroyo-Rabasa }
                                                \date{\vspace{-.2in}}}%

\usepackage{mathtools}
\usepackage{bbm}
\usepackage[dvipsnames]{xcolor}
\usepackage{bm}
\usepackage{enumitem}
\usepackage{mathrsfs}
\usepackage{hyperref}
\hypersetup{
	colorlinks=true,
	citecolor=black,
	linkcolor=black,
	filecolor=black,      
	urlcolor=blue,
}
\usepackage[all]{xy}
\usepackage[textwidth=8em,textsize=small]{todonotes}
\usepackage[multiple]{footmisc}

\newcommand{\Krm}{\mathrm{K}}

\newcommand{\Mrm}{\mathrm{M}}

\newcommand{\Orm}{\mathrm{O}}

\newcommand{\Srm}{\mathrm{S}}

\newcommand{\Acal}{\mathcal{A}}

\newcommand{\Fcal}{\mathcal{F}}

\newcommand{\Kcal}{\mathcal{K}}
\newcommand{\Lcal}{\mathcal{L}}
\newcommand{\Mcal}{\mathcal{M}}

\newcommand{\Qcal}{\mathcal{Q}}

\newcommand{\Xcal}{\mathcal{X}}

\newcommand{\Rbf}{\mathbb{R}}
\newcommand{\Sbf}{\mathbf{S}}

\newcommand{\Cbb}{\mathbb{C}}

\newcommand{\Lbb}{\mathbb{L}}

\newcommand{\Rbb}{\mathbb{R}}

\newcommand{\Tbb}{\mathbb{T}}

\newcommand{\Zbb}{\mathbb{Z}}

\DeclareMathOperator{\SO}{SO}
\DeclareMathOperator{\SL}{SL}
\DeclareMathOperator{\so}{\mathfrak{so}}
\DeclareMathOperator{\sll}{\mathfrak{sl}}
\DeclareMathOperator{\id}{id}
\DeclareMathOperator{\Mat}{Mat}
\DeclareMathOperator{\dev}{dev}

\DeclareMathOperator{\im}{Im}
\DeclareMathOperator{\diam}{diam}

\DeclareMathOperator{\diverg}{div}

\DeclareMathOperator{\dist}{dist}
\DeclareMathOperator{\rank}{rank}

\DeclareMathOperator{\spn}{span}

\newcommand{\set}[2]{\left\{\, #1 \ \textup{{:}}\ #2 \,\right\}}

\newcommand{\dpr}[1]{\langle #1 \rangle}

\newcommand{\dprb}[1]{\bigl\langle #1 \bigr\rangle}

\newcommand{\cl}[1]{\overline{#1}}

\newcommand{\loc}{\mathrm{loc}}
\newcommand{\sym}{\mathrm{sym}}

\newcommand{\toweak}{\rightharpoonup}

\newcommand{\embed}{\hookrightarrow}

\newcommand{\sbullet}{\begin{picture}(1,1)(-0.5,-2)\circle*{2}\end{picture}}
\newcommand{\frarg}{\,\sbullet\,}

\newcommand{\eps}{\epsilon}

\DeclareMathOperator{\Hom}{Hom}
\DeclareMathOperator{\Hol}{Hol}

\newtheorem{theorem}{Theorem}
\newtheorem{lemma}{Lemma}
\newtheorem{definition}{Definition}

\newtheorem{corollary}{Corollary}
\newtheorem{example}{Example}

\newtheorem{remark}{Remark}

\newtheorem{proposition}[theorem]{Proposition}

\newtheorem{thmx}{Theorem}

\newtheorem{notation*}{Notation}
\begin{document}

	\maketitle
		 
\thispagestyle{empty}

\setlength{\abovedisplayskip}{8pt plus 3pt minus 5pt}
\setlength{\belowdisplayskip}{8pt plus 3pt minus 5pt}

	\begin{abstract} Let $1 < p < \infty$ and let $\Omega$ be an open and bounded set of $\R^n$. We establish classical  Korn inequalities
	\[
	\inf_{\substack{v \in L^p(\Omega)\\\Acal v = 0}} \|u - v\|_{W^{k,p}(\Omega)}  \le C \| \Acal u\|_{L^p(\Omega)}
	\]
for all $k$\textsuperscript{th} order operators $\Acal$ satisfying the \emph{maximal-rank} condition. This \emph{new} condition is satisfied by the divergence, Laplacian, Laplace-Beltrami, and Wirtinger operators, among others.  
As such, our estimates generalize Fuchs' estimates for the del-bar operator to maximal-rank operators and to {arbitrary} open sets. 
For domains with sufficiently regular boundary $\partial \Omega$, we are able to construct an $L^p(\Omega)$-bounded projection $P$, onto the kernel of the operator. This projection is shown to satisfy a classical Fonseca--M\"uller projection estimate
\[
	\|u - Pu\|_{L^p(\Omega)}  \le C \| \Acal u\|_{W^{-k,p}(\Omega)}
\]
as well as analogous estimates for higher-order derivatives. As a particular application of our results, we are able to establish a \emph{weak Korn inequality} for general constant-rank operators (by taking the infimum over all $\Acal$-harmonic maps instead of taking it over all $\Acal$-free maps).  Several examples are discussed.  \\

		\vspace{4pt}
		
		\vspace{5pt}
	\end{abstract}
	
	\newpage
	\tableofcontents
\newpage

\section{Introduction} Let $\Omega$ be an open and bounded subset of $\R^n$ and let $k \ge 1$ be an integer. In all the following, our analysis will be restricted to estimates for exponents in the range $1 < p < \infty$.  
We consider a constant coefficient $k$\textsuperscript{th} order homogeneous linear partial differential operator $\Acal$, acting on sufficiently regular maps $u: \Rbb^n \to V$ as 
\begin{equation}\label{eq:A}
\Acal u \, = \, \sum_{|\alpha|=k} A_\alpha D^\alpha u, \qquad A_\alpha \in   \Hom(V,W)\,,
\end{equation} 
where $V,W$ are finite-dimensional $\R$-spaces,  
$(\alpha_1,\dots,\alpha_n)$ is a multi-index of non-negative integers with modulus $|\alpha| = \alpha_1 + \dots + \alpha_n$, and $D^\alpha$ is the composition of the distributional partial derivates $\partial_1^{\alpha_1}\circ \cdots \circ \partial_n^{\alpha_n}$. 
Up to a linear isomorphism, the reader may think of $V$ and $W$ as $\Rbf^M$ and $\Rbf^N$ respectively, in which case $\Acal$ is precisely a system of $N \times M$ partial differential equations. When the operator is homogeneous as in our assumptions, the Fourier transform establishes a one-to-one correspondence between the operator $\Acal$ and its associated \emph{principal  symbol} map $A: \Rbf^n \to \Hom(V,W)$, which is the $k$-homogeneous tensor-valued polynomial defined by
\[
A(\xi) \coloneqq \sum_{|\alpha| = k} A_\alpha\xi^\alpha, \qquad \xi^\alpha \coloneqq \xi_1^{\alpha_1} \cdots \xi_n^{\alpha_n}, \qquad \xi \in \Rbf^n\,.
\]
This work is primarily concerned with the study of Sobolev-distance estimates of the form
\begin{equation}\label{eq:start}
	\inf_{\substack{v \in L^p(\Omega),\\ \Acal v = 0}} \|u - v\|_{W^{k,p}(\Omega)} \le C \|\Acal u\|_{L^p(\Omega)}\,,
\end{equation}
for operators satisfying the \emph{maximal-rank condition}
\begin{equation}
\im A(\xi) = \im A(\eta) \quad \text{for all $\xi, \eta \in \R^n - \{0\}$}\,.
\end{equation} 
Notice that all maximal-rank operators satisfy the well-known (see, e.g.,~\cite{kato,murat1981compacite,SW1}) \emph{constant-rank} condition
\begin{equation}\label{eq:cr}
	\rank \, ( \ker A(\xi) )= \text{const.} \quad \text{for all $\xi \in \R^n - \{0\}$\,.}
\end{equation}

\subsection*{Background theory and motivation}The association between operators and their principal symbol polynomials is often relevant and interesting in the sense that certain functional properties of \guillemotleft~$\Acal$~\guillemotright\, are transferred to certain algebraic properties of the symbol \guillemotleft~$A$~\guillemotright\, and vice-versa. In full-space, the classical  Calder\'on--Zygmund estimates (see~\cite{CZ}) for constant-coefficient elliptic operators  imply that if $\Acal$ is elliptic, i.e, 
\[
	\ker A(\xi) = \{0_V\} \quad \text{for all $\xi \in \R^n - \{0\}$}\,,
\]
then the norms $\|D^ku\|_{L^p}$ and $\|\Acal u\|_{L^p}$ are equivalent in $C_c^\infty(\R^n;V)$; it is well-known to experts that the ellipticity condition is also necessary for the equivalence of the norms (see for instance~\cite{vans2013jems}). 
The generalization of this result to non-elliptic operators is one of the cornerstones of the modern compensated compactness theory:
\begin{thmx}[Fonseca and M\"uller, 1999\,;\, Guerra and Raita, 2020]\label{thm:FM99} Let $\Acal $
be a homogeneous differential operator of order $k$ on $\R^n$ as in~\eqref{eq:A}. The following are equivalent: (\,\footnote{Originally, the sufficiency \ref{A1} of the constant-rank condition is due to Fonseca and M\"uller, for first-order operators~\cite{fonseca1999quasi} and periodic maps; the statement for higher-order systems is a mere observation of this and is recorded in~\cite{advances}. In full space, the sufficiency was recorded by Raita in the unpublished note~\cite{BR}. The necessity \ref{A2} of the constant-rank condition in full-space was recently established by Guerra and Raita in~\cite{guerra2020necessity}.})
\begin{enumerate}[label=\arabic*., ref=\arabic*, left=10pt]
\item \label{A1}  There exists $r \in \Zbb$ such that
\[
	\rank A(\xi) = r \qquad \text{for all $\xi \in \R^n - \{0\}$}\,.
\]
\item \label{A2} There exists a constant $C = C(p, A)$ such that 
\[
	\|D^k (u -Pu)\|_{L^p(\R^n)} \le C  \|\Acal u\|_{L^p(\R^n)}
\]
for every $u \in C^\infty_c(\R^n;V)$. Here, $P$ is the $L^p$-extension of the $L^2$-projection onto the distributional kernel of $\mathcal A$ restricted to the space $C^\infty_c(\R^n;V)$.  
\end{enumerate}
\end{thmx} 
\begin{remark}Analogous \emph{projection estimates} hold if instead of considering functions defined over $\Rbf^n$, one considers functions over the $n$-dimensional torus $\Tbb^n = \Rbb^n / \Zbb^n$; the crucial element of the proof being the availability of Fourier coefficient decompositions.
\end{remark}  In this vein, one may ask whether similar estimates hold on domains of $\R^n$. To this end, one considers $\Acal $ as an unbounded linear operator 
\[
	\Acal : D(\Acal) \subset L^p(\Omega;V) \longrightarrow L^p(\Omega;W)\,.
\]
Denoting its nullspace by $N_p(\Acal,\Omega)$, one could expect the constant-rank condition~\eqref{eq:cr} to be a sufficient (and necessary) condition for the existence of bounded linear projection $P: L^p(\Omega;V) \twoheadrightarrow N_p(\Acal,\Omega)$ satisfying 
\begin{equation}\label{eq:projection_estimate}
	\|u -Pu\|_{W^{k,p}(\Omega)} \le C  \|Au\|_{L^p(\Omega)}  \quad \text{for all $u \in D(\Acal)$.}
\end{equation}
It turns out that proving projection estimates on a domain $\Omega \subset \Rbf^n$ is considerably more challenging than working on $\Rbf^n$ or $\Tbb^n$. This stems from the lack of Fourier transform arguments and the fact that, even for balls, there are no suitable trace or extension operators for general constant-rank operators.  
In fact, it is not known whether the constant-rank property is a sufficient condition for the validity of the  Sobolev-distance estimate (\,\footnote{For $p= 2$, the projection estimate~\eqref{eq:projection_estimate} and the Sobolev-distance estimate~\eqref{eq:distance} are equivalent.}\,)\begin{equation}\label{eq:distance}
	\inf_{v \in N_p(\Acal,\Omega)} \|u - v\|_{W^{k,p}(\Omega)} \le C \|\Acal u\|_{L^p(\Omega)} \quad \text{for all $u \in D(\Acal)$\,.} 
\end{equation}

\begin{remark}[Korn's inequality]
On \emph{Poincar\'e domains} (domains where the classical Poincar\'e inequality is satisfied) such as Lipschitz, H\"older or Jones domains, the Sobolev-distance estimate is equivalent to the Korn inequality 
\[
	\inf_{v \in N_p(\Acal,\Omega)}\|D^k(u - v)\|_{L^p(\Omega)} \le C \|\Acal u\|_{L^p(\Omega)}\,.
\]
\end{remark}

There are a few well-known instances (on sufficiently regular domains) for which Sobolev-distance estimates are known to hold. For the $k$\textsuperscript{th} order gradient operator, the  Deny--Lions Lemma establishes that
 \[
 \min_{p \in P_{k-1}(V)}\|u - p\|_{W^{k,p}(\Omega)} \le C \|D^ku\|_{L^p(\Omega)}\,,
 \]
where $P_r(V)$ denotes the space of polynomials on $n$ variables, with $V$-valued coefficients, and order at most $r$.
For the symmetric gradient operator $\varepsilon (u) = \frac 12 (Du + Du^T)$, 
 the estimate follows directly from Korn's second inequality (cf. Sect.~\ref{Sec:Korn}). 
Korn estimates for the deviatoric operator $\varepsilon^D (u) = Eu - n^{-1} (\diverg u) I_n$, which also acts on vector fields, were established for $n \ge 3$ in~\cite{Diening1} (although the result had been known long before as it is discussed next). The operators in these three examples \guillemotleft~$D^k, \eps, \eps^D$~\guillemotright\, belong to a class of ``very elliptic'' operators, introduced by Aronszajn~\cite{ar} and Smith~\cite{smith1,smith2},  with  finite-dimensional \emph{distributional nullspace} 
\[
N(\Acal,\R^n) \coloneqq \set{u \in \mathscr D'(\R^n;V)}{\Acal u = 0}. \quad (\,\footnote{Smith's original definition~\cite{smith1} requires the complexification of the principal symbol \guillemotleft~$A(\xi)_{\Cbb} \in \Hom_\C(V,W)$~\guillemotright\, to be one-to-one for all non-zero complex frequencies $\xi \in \Cbb^n - \{0\}$; which is why these operators have been recently coined (see~\cite{Diening}) under the ad-hoc name of \emph{complex-elliptic operators}. Smith himself showed that complex ellipticity is equivalent to the \emph{FDN property} (finiteness of its distributional nullspace).}\,)
\] 
For operators in this class, a linear trace functional exists, and therefore, stronger estimates hold on Lipschitz domains~\cite{smith2} (even on domains satisfying considerably milder regularity assumptions, cf.~\cite{diening2021sharp}):

\begin{thmx}[Smith, 1970\,;\, Diening and  Gmeineder, 2021]\label{thm:Smith}
	Let $\Omega \subset \R^n$ be a Lipschitz domain (or Jones) domain and let $\Acal $
be a homogeneous differential operator of order $k$ on $\R^n$. The following are equivalent: ( \footnote{The equivalence was proved by Smith for $1 < p < \infty$ on Lipshchitz domains. Diening and Gemeineder established the same for Jones domains.} )
\begin{enumerate}[label=\arabic*., left=10pt]
\item \label{B1} $\dim \, N(\Acal,\R^n) < \infty$\,.
\item There exists a constant $C = C(p,\Omega,A)$ such that 
\[
	\|u\|_{W^{k,p}(\Omega)} \le C \left(\|u\|_{L^p(\Omega)} + \|\Acal u\|_{L^p(\Omega)} \right) \qquad \text{for all $u \in D(\Acal)$}\,.
\]
\item \label{B3} There exists a bounded linear projection $P : L^p(\Omega;V) \longrightarrow  N(\Acal,\R^n)$ such that
\[
	\|u - Pu\|_{W^{k,p}(\Omega)} \le C(p,\Omega,A)  \|\Acal u\|_{L^p(\Omega)}
\]
for all $u \in L^p(\Omega;V)$ with $\Acal u \in L^p(\Omega;W)$.
\end{enumerate}
\end{thmx}
\begin{remark}
A more comprehensive list of equivalencies, some of which are well-known to experts, are recorded in Prop.~\ref{prop:CE}.
\end{remark}
\begin{remark}[Necessity of the restriction $1 < p<\infty$]
For $p \in \{1,\infty\}$, there are no nontrivial Korn or Sobolev estimates for elliptic operators. Ornstein~\cite{Ornstein} established the result for equations and $p=1$ (the proof for systems is contained in~\cite{KK}). The case for $p=\infty$ was established in~\cite{Lew}.
\end{remark}

\begin{remark}[Maximal-rank is not FDN]
The nullspace $N(\Acal,\R^n)$ of a nontrivial maximal-rank operator $\Acal$ is infinite dimensional (see Rmk.~\ref{rem:fdn}).
\end{remark}

In the context of Korn and projection estimates, Theorem~\ref{thm:Smith} accounts only for a particular case. Indeed, in general, the restriction $N(\Acal,\R^n)|_\Omega$ is strictly smaller than $N_p(\Acal,\Omega)$ for operators with infinite-dimensional nullspace~(cf. Prop.~\ref{prop:CE}). Therefore, while sufficient, the property $\dim N(\Acal,\R^n) < \infty$ (or the existence of traces on $\partial \Omega$) is \emph{strictly unnecesary}  
 for the validity of~\eqref{eq:projection_estimate} or~\eqref{eq:distance}. In the realm of operators with infinite-dimensional kernel, it seems the only previous result is due to Fuchs~\cite{Fuchs}, who established Sobolev-distance estimates for the Wirtinger derivative operator \guillemotleft~$\partial_{\bar z}$~\guillemotright\, on regular subdomains of the complex plane: (\,\footnote{The Wirtinger derivative $\partial_{\bar z}$ is, in fact, one of the simplest prototypes  of a first-order operator with infinite-dimensional nullspace:  $N(\partial_{\bar z})$ coincides with the space $\Hol(\Cbb)$ of holomorphic maps.}\,)
\begin{thmx}[Fuchs, 1997]\label{thm:Fuchs} Let $\omega \subset \Cbb$ be a Lipschitz subdomain. Then, there exists a constant $C = C(p,\omega)$ such that 
\[
\inf_{\text{$v \in \Hol(\omega)$}} \|u - v\|_{W^{1,p}(\omega)} \le C\|\partial_{\bar z} u\|_{L^{p}(\omega)}
\] 
for all $u \in W^{1,p}(\omega;\Cbb)$. 
\end{thmx} 
\subsection*{Summary of results} Before this work appeared, there was no consensed understanding as to what (symbolic) properties of the Wirtinger derivative $\partial_{\bar z}$ were responsible for the validity of Fuchs' estimate.  Given that  $\partial_{\bar z}$ is an operator with rich analytical and geometrical properties, it had remained unclear whether the estimate stemmed from more particular properties than ellipticity. The results presented here establish that $\partial_{\bar z}$ belongs to a large class of constant-rank (possibly non-elliptic) operators, with {infinite-dimensional nullspace}, satisfying Sobolev-projection and Sobolev-distance estimates on {arbitrary} open (possibly disconnected) domains $\Omega$ of $\R^n$. \\

We will divide the exposition of our results as follows:

\begin{itemize}[left= 0pt,label=\S]
\item \textbf{Maximal-rank operators on arbitrary open domains.} In Theorem~\ref{thm:2}, we prove that the maximal-rank property is  sufficient for the validity of (unbounded) linear projection estimates on possibly irregular domains $\Omega \subset \R^n$. This conveys (see Corollary~\ref{cor:weak}) a hierarchy of lower-order distance estimates. 
\item \textbf{Maximal-rank operators on regular domains.} On domains with sufficiently regular boundary, we are able to improve the unbounded linear projection estimate to a linear bounded projection estimate. This, in turn, conveys a version of Theorem~\ref{thm:FM99} on regular domains. Lower-order projection estimates are also established. 
\item \textbf{Weak estimates for constant-rank operators on domains.} Lastly, we discuss the validity of \emph{weaker} distance estimates on arbitrary domains. In Theorem~\ref{thm:Z} we establish that the Sobolev distance from $N_p(\Delta_\Acal,\Omega)$ is bounded in terms of $\Acal u$. (\,\footnote{Weak in the sense that $N_p(\Delta_\Acal,\Omega)$ is usually considerably larger than $N_p(\Acal,\Omega)$.}\,)
\end{itemize} 
\subsection*{Comments about the proof} The main ingredient of the proof(s) is contained in Lemma~\ref{lem:existence}, which can be labeled as an \emph{existence and regularity result for maximal-rank operators}. There, we show that if $\Acal$ has a maximal rank, then there exists a kernel 
\[
K \in L^1_\loc(\R^n;\Hom(W,V))
\]
 acting as a \emph{true fundamental solution}  
	\[
		\Acal K = \delta_{0_W} \quad \text{on $\R^n$\,,}
	\]
	where $\delta_{0_W}$ is the Dirac mass at $0 \in W$. This gives rise to a convolution-type solution operator on $\Omega$, for which we can show classical Sobolev estimates. Although similar solution kernels also exist for general constant-rank operators (see, e.g.,~\cite{BVS,BR,raita2019potentials}), these, in general, only satisfy the weaker fundamental identity
	\[
		\Acal (K \star  \Acal u) = u \quad \text{for all $u \in C_c^\infty(\Omega;V)$}\,.
	\] 
It is precisely for this step that the maximal-rank assumption is crucially used to guarantee there exists a convolution solution operator $\Acal^{-1} : L^p(\Omega;W) \to W^{k,p}(\Omega;V)$  such that, for \emph{any} right-hand side $f \in L^p(\Omega;W)$, the element $\Acal^{-1}[f] \in L^p(\Omega;V)$ solves the system
	\[
		\Acal u = f \qquad \text{on $\Omega$\,,}
	\]
	together with the Sobolev estimate
	\[
		\|\Acal^{-1}[f]\|_{W^{k,p}(\Omega)} \le C \|f\|_{L^{p}(\Omega)}\,.
	\]
	Then, one simply defines the projection $Tu \coloneqq u - \Acal^{-1}[\Acal u]$. The regularity of the domain is only used to deduce the linearity of the inverse operator $\Acal^{-1}$, when considering spaces of maps with less regular properties (such as $L^p$ or $W^{r,p}$ with $r < k$).


\section{Main results}

\subsection[Estimates on domains]{Estimates on arbitrary domains} 


Our first result establishes that if $\Acal$ has maximal rank, 
 then Sobolev estimates (away from $N_p(\Acal,\Omega)$) on the domain
 \[
 W^{\mathcal A,p}(\Omega) \coloneqq \set{u \in L^p(\Omega;V)}{\Acal u \in L^p(\Omega;W)}\,,
 \] 
 hold for arbitrary sub-domains of $\Omega \subset \Rbb^n$: 

\begin{theorem}[Unbounded projection on arbitrary domains]\label{thm:2} Let $\Acal $ be a maximal-rank operator of order $k$, from $V$ to $W$. 	Then, there exists an unbounded linear map $$T : W^{\mathcal A,p}(\Omega) \subset L^p(\Omega;V) \longrightarrow L^p(\Omega;V)$$
	satisfying the following properties: if $u \in W^{\Acal,p}(\Omega)$, then
	\begin{enumerate}[label=\arabic*., left=15pt]
	\item $T$ is a projection onto $N_p(\Acal,\Omega)$, that is, 
	\[
	R(T) = N_p(\Acal,\Omega) \quad \text{and} \quad 
	T( T u) = T u\,.
	\]
\item There exists a constant $C$ depending solely on  $p,\diam \Omega$, $A$ such that
	\[
		\|T u\|_{L^p(\Omega)} \le \|u\|_{L^p(\Omega)} + C\, \|\Acal  u\|_{L^{p}(\Omega)},
	\]
 \item and
	 \[
		\| u - Tu \|_{W^{k,p}(\Omega)} \le C\,  \|\Acal  u\|_{L^{p}(\Omega)}\,.
	\]
\end{enumerate}
\end{theorem}
\begin{remark}
When $\Omega$ is a Lipschitz domain,  the projection $T$ can be extended to a bounded linear map on $L^p(\Omega;V)$; see Theorem~\ref{thm:FM} below.
\end{remark}
\begin{corollary}[Compactness] Assume that $\Omega$ is a Lipschitz domain and set $P \coloneqq \id - T$. Let $\Xcal \subset L^p(\Omega;V)$ be a family  satisfying
\[
	\sup_{u \in \Xcal} \|\Acal u\|_{L^p(\Omega)} < \infty\,.
\]
Then $P[\Xcal]$ is pre-compact in $L^p(\Omega;V)$ (and in $W^{k-1,p}(\Omega;V)$ whenever $k > 1$).
\end{corollary}

Another implication is the following decomposition:
\begin{corollary}[Helmholtz-type decomposition]\label{cor:H1} The space $W^{\Acal,p}(\Omega)$ decomposes as a topological sum (of closed subspaces with trivial intersection)
\[
	W^{\Acal,p}(\Omega) = N_p(\Acal,\Omega) \oplus \Acal^{-1}[L^p(\Omega;W)]\,,
\]
where $\Acal^{-1} : L^p(\Omega;W) \to W^{k,p}(\Omega;V)$ is the solution operator constructed in Lemma~\ref{lem:existence}.
That is, every $u \in W^{\Acal,p}(\Omega)$ may be uniquely decomposed as
\[
	u = v + w, \qquad v \in N_p(\Acal,\Omega),\quad w \in \Acal^{-1}[L^p(\Omega;V)]\,.
\]	
Moreover, in this case, $w$ satisfies the Sobolev estimate
\[
	\|w\|_{W^{k,p}(\Omega)} \le C(p,\diam\Omega,A)\, \|\Acal u\|_{L^p(\Omega)}\,.
\]
\end{corollary}

\subsection*{Lower-order distance estimates}The projection estimates contained in Theorem~\ref{thm:2} convey a hierarchy of lower-order distance estimates in terms of suitable negative Sobolev norms of $\Acal u$.  To make this statement precise, let $\ell > 0$ be a positive integer and let us recall that $W^{\ell,p}_0(\Omega)$ is defined as the closure of $C^\infty_c(\Omega)$ under the $W^{\ell,p}$-norm. Following standard notation, we write $W^{-\ell,p}(\Omega)$ to denote the dual of $W^{\ell,q}_0(\Omega)$, where $p^{-1} + q^{-1} = 1$. Notice that when $\Omega$ is a bounded open set, then, by Poincare's inequality,  the classical Sobolev norm $\|\frarg\|_{W^{\ell,p}(\Omega)}$ and the homogeneous norm 
\[
\|u\|_{\dot W^{\ell,p}(\Omega)} \coloneqq \left(\sum_{|\alpha| = \ell} \|D^\alpha u\|_{L^p(\Omega)}^p\right)^\frac{1}{p}
\]
are equivalent norms {of} $W^{\ell,p}_0(\Omega)$. Therefore, also $W^{-\ell,p}(\Omega)$ coincides with the homogeneous negative Sobolev space $\dot W^{-\ell,p}(\Omega)$, which is defined as the dual of  $W^{\ell,p}_0(\Omega)$ when endowed with the homogeneous norm. In fact, their respective  norms are equivalent:
\[
	\|\frarg\|_{W^{-\ell,p}(\Omega)}  \, \simeq \, \|\frarg\|_{\dot W^{-\ell,p}(\Omega)}.
\]
For completeness, we set $W^{0,p}(\Omega) = \dot W^{0,p}(\Omega) \coloneqq L^p(\Omega)$.  With these considerations in mind, we state the following distance estimates for maximal-rank operators.
\begin{corollary}[Distance estimates]\label{cor:weak}
	 Let $0 \le r \le k$ be an integer and let $\Acal $ be a maximal-rank operator of order $k$, from $V$ to $W$. Then,
	 \[
	\min_{v \in N_p(\Acal,\Omega)} \| u - v\|_{W^{k -r,p}(\Omega)} \le C(p,r,\diam\Omega, A) \, \|\Acal u\|_{\dot W^{-r,p}(\Omega)}
	\]
for all $u \in L^p(\Omega;V)$ with $\Acal u \in \dot W^{-r,p}(\Omega;W)$.
	\end{corollary}
%
	
	\subsection{Bounded projections on regular domains} In general, the inverse operator $\Acal^{-1}$ (and hence also the projection operator $T$ in Theorem~\ref{thm:2}) may only be extended by a possibly non-linear map satisfying similar Sobolev estimates. If, however, $p = 2$, or if $\Omega$ is a sufficiently regular and connected domain, then one can extend $\Acal^{-1}$ linearly to a solution operator on $W^{-k,p}(\Omega;V)$. In the first case, this follows from the theory of Hilbert spaces, while in the latter case, it follows from the classical $L^p$-regularity theory for the higher-order Dirichlet problem on regular domains. In either case, we get the following generalization of the Fonseca--M\"uller projection estimates:
	\begin{theorem}[Fonseca--M\"uller projection estimate on domains]\label{thm:FM}Le $\Acal $ be a maximal-rank operator of order $k$, from $V$ to $W$.
Further, assume that $p =2$ or that $\Omega$ is a connected domain with smooth boundary  $\partial \Omega$. Then, the projection $T$  from Theorem~\ref{thm:2} extends to a {bounded} linear projection $$\cl T : L^p(\Omega;V) \longrightarrow L^p(\Omega;V),$$ satisfying the projection estimate 
\begin{align*}
\| u - \cl T u\|_{L^p(\Omega)} & \le C(p,\diam\Omega,\partial\Omega,A)\, \|\Acal u\|_{\dot W^{-k,p}(\Omega)}\,. 
\end{align*}
More generally, if $0 \le r \le k$ is an integer, then \[
\| u - \cl T u\|_{W^{k-r,p}(\Omega)} \le C(p,\diam\Omega,\partial \Omega,A,r) \, \|\Acal u\|_{\dot W^{-r,p}(\Omega)}
\]
for all $u \in L^p(\Omega;V)$ with $\Acal u \in \dot W^{-r,p}(\Omega)$. ( \footnote{Notice that $\Acal u \in \dot W^{-k,p}(\Omega)$ whenever $u \in L^p(\Omega;V)$.}\,)
\end{theorem}


\begin{corollary}\label{cor:R} Under the same assumptions of the Theorem~\ref{thm:FM}, the space $L^p(\Omega; V) $ decomposes as a topological sum (of closed subspaces)
\[
	L^p(\Omega;V) = N_p(\Acal,\Omega) \oplus \Acal^{-1}[\dot W^{-k,p}(\Omega;W)], 
\]
where $\Acal^{-1} : \dot W^{-k,p}(\Omega;W) \to L^p(\Omega;V)$ is the extended linear solution operator  from Lemma~\ref{lem:existence}.
More precisely, every $u \in L^p(\Omega; V)$ may be uniquely decomposed as 
\[
	u = v + w, \qquad v \in N_p(\Acal,\Omega), \quad w \in  \Acal^{-1}[\dot W^{-k,p}(\Omega;W)],
\]
with 
\[
 \|w\|_{L^p(\Omega)} \le C(p,\diam \Omega,\partial \Omega,A)\,\|\Acal u\|_{\dot W^{-k,p}(\Omega)}\,.
\]
\end{corollary}

\begin{remark}
These two results hold under significantly milder regularity assumptions on $\partial \Omega$. However, we will avoid being more precise about it as it falls out of our primary objective. 
\end{remark}
	

\subsubsection{Estimates for $L^1$-gradients} We can make use of the previous results to establish certain \emph{sub-critical} Sobolev estimates when $\Acal u$ is either representable by an integrable map or a Radon measure. Let us first recall that the total variation of a distribution $\sigma \in \mathscr D'(\Omega;W)$, on a Borel subset $U \subset \Omega$, is defined as the (possibly infinite) non-negative quantity
\[
|\sigma|(U) \coloneqq \sup\set{\sigma[\varphi]}{\varphi \in C_c^\infty(U;W),\|\varphi\|_\infty \le 1}.
\]
The space 
\[
	\Mcal_b(\Omega;W) \coloneqq (C_0(\Omega;W))^*,
\]
of finite $W$-valued Radon measures over $\Omega$, coincides with the subspace of distributions $\sigma \in \mathscr D'(\Omega;W)$ with finite total variation $|\sigma|(\Omega)$.
Notice that if $\sigma \in L^1(\Omega;W)$, then the absolutely continuous measure $\sigma \, \mathscr L^n$ has bounded variation on $\Omega$ and in fact $\|\sigma\|_{L^1(\Omega)} = |\sigma\,\mathscr L^n|(\Omega)$.   Morrey's embedding and our previous results imply the following projection estimates for functions whose $\Acal$-gradient is a measure:
\begin{corollary}\label{cor:measures}Let $1 < q < \frac {n}{n-1}$ and let $\Acal$ be a maximal-rank operator of order $k$, from $V$ to $W$.
Then, 
\[
	\min_{v \in N_q(\Acal,\Omega)} \|u - v\|_{W^{k-1,q}(\Omega)} \le C(n,q,\diam \Omega,A) |\Acal u|(\Omega)
\]
for all $u \in L^q(\Omega;V)$. If moreover $\Omega$ is connected with smooth boundary  $\partial \Omega$, 
then also
\[
	\|u - \cl T u\|_{W^{k-1,q}(\Omega)} \le C(n,q,\diam \Omega,\partial\Omega,A) |\Acal u|(\Omega) 
\]
for all $u \in L^q(\Omega;V)$. ( \footnote{Here, $\cl T$ is the projection from  Theorem~\ref{thm:FM} with exponent $p = q$.}\textsuperscript{,}\footnote{The critical $W^{k-1,n/(n-1)}$ estimate fails for all nontrivial maximal-rank operators; for this we refer the reader to Rem.~\ref{rem:fdn} and to the main result in~\cite{vans2013jems}.}\,)
\end{corollary}

\subsection{Weak estimates for constant-rank operators} If we assume that $\Acal$ is a constant-rank operator (but not necessarily of maximal rank), we are currently only able to prove the validity of \emph{weaker} distance estimates, which require one to remove a subspace of the space of all $\Acal$-harmonic maps.  To make this precise, let us introduce the generalized $\Acal$-Laplacian operator (see Sect.~\ref{sec:2})
 	\[
 	\Delta_\Acal  \coloneqq \Acal ^* \circ \Acal .
 	\]
The statement is the following:
\begin{theorem}\label{thm:Z} Let $0 \le r \le k$ be an integer and let  $\Acal$ be an operator of order $k$ on $\Rbf^n$, from $V$ to $W$. Further assume that $\Acal$ satisfies the constant-rank property
 	\[
 \forall \xi \in \Rbf^n - \{0\},  \qquad\rank A(\xi) = const.
 	\]
	 Then, 
 \[
 	\inf_{v \in N_p(\Delta_{A},\Omega)} \|u - v\|_{W^{k-r,p}(\Omega)} \, \le C(p,r,\diam \Omega,A) \,\|\Acal u\|_{\dot W^{-r,p}(\Omega)}
 \]
 for all $u \in L^p(\Omega;V)$.
\end{theorem}

\section{Preliminaries}\label{sec:2}
As already mentioned above, $\Omega$ will always denote an arbitrary bounded and open subset of $\Rbf^n$.  If $u \in L^p(\Omega;V)$ and $\sigma \in \mathscr D'(\Omega;W)$, the system  
\[
\Acal u = \sigma \quad \text{on $\Omega$}
\]
shall always be understood in the sense of distributions, that is,
\begin{align*}
\sigma[\phi] = \int_\Omega u \cdot \Acal^* \phi \qquad \text{for all $\phi \in C^\infty_c(\Omega;W)$,}
\end{align*}
where 
\[
\Acal^*  \coloneqq (-1)^k\sum_{|\alpha| = k}   (A_\alpha)^* D^\alpha 
\] 
is the formal $L^2$-adjoint of $\Acal$. In this case, we write  $\Acal u \in L^p(\Omega;W)$ provided that $\sigma$ is representable by a $p$-integrable map, which we shall also identify with $\Acal u$. With this convention, we may rigorously define the  distributional null-space $N(\Acal;\Omega) = \set{u \in \mathscr D'(\Omega;V)}{\Acal u = 0}$ and its subspace 
\[
N_p(\Acal;\Omega) \coloneqq N(\Acal;\Omega) \cap L^p(\Omega;V) \le L^p(\Omega;V),
\]
consisting of all $p$-integrable $\Acal$-free maps on $\Omega$. 
One can see (e.g., ~\cite[Sect. 2.5]{advances}), by a simple application of the Fourier transform and a density argument, that the subspace of all $\Acal$-gradients in Fourier space
\[
	 \spn \set{\im A(\xi)[v]}{\xi \in \Rbf^n, v \in V},
\]
coincides with the point-wise essential range of $\Acal$ given by
\[
 W_\Acal = \mathrm{clos}_{W}\set{\Acal u(x)}{x \in \Rbf^n, u \in C_c^\infty(\Rbf^n;V)}.
\]

\subsection{FDN vs. maximal-rank} 

Next, we record the following equivalences for FDN operators (some of which are already well-established results): 
\begin{proposition}\label{prop:CE} 
	Let $\Omega \subset \Rbf^n$ be a Lipschitz (or Jones) domain and let $\Acal$ be a  homogeneous $k$\textsuperscript{th} order linear operator on $\Rbf^n$, from $V$ to $W$. The following are equivalent: 
	\begin{enumerate}[label=\arabic*., ref=(\arabic*), left=10pt]
\item $\dim N(\Acal,\R^n) < \infty$, \label{1}
\item $A(\xi)_\C : \Cbb \otimes V \to \Cbb \otimes W$  is one-to-one for all non-zero $\xi \in \Cbb^n$, \label{2}
\item $N(\Acal,\R^n)$ is a space of $V$-valued polynomials on $n$-variables, \label{3}
\item the restriction of $N(\Acal,\R^n)$ to $\Omega$ coincides with $N_p(A,\Omega)$, \label{4}
\item \label{5}
\[
	\|u\|_{W^{k,p}(\Omega)} \le C(p,\Omega,A) \left(\|u\|_{L^p(\Omega)} + \|Au\|_{L^p(\Omega)} \right) \quad \text{for all $u \in D(\Acal)$}\,,
\]
\item $N(\Acal,\R^n)|_\Omega \subset W^{k,p}(\Omega;V)$ and \label{6}
	\begin{equation*} 
	\min_{v \in N(\Acal)}  \|u - v\|_{W^{k,p}(\Omega)} \le C(p,\Omega,A) \|A u\|_{L^{p}(\Omega)} \qquad \text{for all $u \in D(\Acal)$}\,,
	\end{equation*}
		\item \label{7} $N(\Acal,\R^n)|_\Omega \subset W^{k,p}(\Omega)$ and there exists at least one  continuous linear projection $$P : W^{k,p}(\Omega;V) \longrightarrow N(\Acal,\R^n)|_\Omega\,.$$ Moreover, for {any} such projection it holds
\[
	\|u - Pu\|_{W^{k,p}(\Omega)} \le C  \|Au\|_{L^p(\Omega)} \qquad \text{for all $u \in D(\Acal)$}.
\]
In this case, the constant $C$ also depends on $\|P\|_{W^{k,p} \to W^{k,p}}$.
\end{enumerate}
\end{proposition}
\emph{Proof.} \begin{proof}The equivalence $\ref{1} \Leftrightarrow \ref{2} \Leftrightarrow \ref{3}$ and $\ref{3} \Leftrightarrow \ref{5}$ (for Lipschitz domains) are contained in~\cite{smith2}; $\ref{1} \Leftrightarrow (4) \Leftrightarrow (5)$ for Jones domains is contained in~\cite{Diening}, as well as a part of $(1) \Rightarrow (7)$ for a particular projection. Clearly $\ref{7}\Rightarrow \ref{5}$ is trivial. We are only left to show that $\ref{5} \Rightarrow \ref{6} \Rightarrow \ref{7}$. Notice that $\ref{5}$ has the form
\[
	\|u\|_E \le C \left(\|Tu\|_F + \|Pu\|_G\right),
\]	
where $E = W^{k,p}(\Omega),F =L^p(\Omega)$ are Banach, $\Acal  = T \in \Lcal(E,F)$, $G = L^p(\Omega)$ is Banach and $P = \id \in \Kcal(E,G)$; here we have used that $\Omega$ is a Jones domain so that Rellich's theorem holds. By 6.9 in~\cite{Brezis}, it follows that $R(T) = R(A)$ is closed in $L^p(\Omega)$. 
In particular, $A : D(\Acal)/N(\Acal)|_\Omega \to R(A)$ is an isomorphism, and by the open mapping theorem, we get the desired bound
\[
	\min_{v\in N(\Acal)} \|u - v\|_{L^p(\Omega)} \coloneqq \|u\|_{L^p(\Omega)/N(\Acal)} \le C \|Au\|_{L^p(\Omega)}.
\]	
This proves (6). Lastly, we show that $\ref{6} \Rightarrow \ref{7}$: Since the Sobolev norm and $L^p$ norms are equivalent on $N(\Acal,\R^n)|_\Omega$, it follows from Rellich's theorem that $N(\Acal,\R^n)|_\Omega$ compact and therefore finite-dimensional. Therefore, there exists at least one continuous linear projection $Q: W^{k,p}(\Omega;V) \to N(\Acal)|_\Omega$. Now, let $P$ be any such projection. It follows that
\[
	\|u - Pu\|_{W^{k,p}(\Omega)} \le \|u - \tilde w_0\|_{W^{k,p}(\Omega)} + \|\tilde w_0 - Pu\|_{W^{k,p}(\Omega)},
\]	
where $w_0 \in W^{k,p}(\Omega) \cap N(\Acal,\R^n)$ is the minimizer of the infimum in (6) and $\tilde w_0 = \chi_\Omega w_0$. Since $P\tilde w_0 = \tilde w_0$, we get directly from (6) and the continuity of $P$ that
\begin{align*}
	\|u - Pu\|_{W^{k,p}(\Omega)} & \le C\|\Acal u\|_{L^p(\Omega)} + \|P\|_{W^{k,p} \to W^{k,p}} \|\tilde w_0 - u\|_{W^{k,p}(\Omega)} \\
	& \le C' \|\Acal u\|_{L^p(\Omega)}.
\end{align*}
This proves \ref{7}. 
\end{proof}
\begin{remark}[Maximal-rank vs FDN]\label{rem:fdn}In general, coercive estimates on domains fail for all non-trivial operators of maximal rank. In fact, 
\[
	\Acal \; \text{is a maximal-rank operator} \quad \Longrightarrow \quad \dim N(\Acal,\Omega) = \infty.
	\]
This is in line with the fact that the validity of strong projection estimates is independent of the validity of coercive inequalities on $\Omega$ (cf.~\ref{5} in Proposition~\ref{prop:CE}) or the existence of suitable linear trace operators on $\partial \Omega$ (cf.~\cite{Diening,GR}). The implication above follows easily from~\cite[Prop. 1.2]{GR}, where it has been shown that $\dim N(\Acal) < \infty$ implies  that $\Acal$ is \emph{canceling} (in the sense of Van Schaftingen~\cite{vans2013jems}):
	\[
	\mathscr W_\Acal \coloneqq \bigcap_{\xi \in \R^n - \{0\}} \im A(\xi) =\{0_W\}.
	\]
However, the cancellation property fails for all non-trivial maximal-rank operators because, by definition, these satisfy $\mathscr W_\Acal = W_\Acal \neq \{0_W\}$.
\end{remark}

\subsection{Properties of negative Sobolev spaces}  Let us introduce the higher-order divergence operator: for a tensor $F = (F_\beta)_{|\beta| = k} \in L^p(\Omega;V)^{\binom{N + k -1}{k}}$, the $k$\textsuperscript{th} order divergence operator is defined as the distributional operator
\[
	\diverg^k F \coloneqq \sum_{|\beta| = k} D^\beta F_\beta.
\]
With this in mind, we give a short proof of a well-known representation of homogeneous negative spaces as quotient spaces of $L^p$-spaces.

\begin{proposition} The map $S : L^p(\Omega)^{\binom{N + k - 1}{k}}/N_p(\diverg^k,\Omega) \longrightarrow \dot W^{-k,p}(\Omega)$ given by 
\[
	S[g] = (-1)^k\diverg^k g,
\]
is a linear bijective isometry. 
\end{proposition}
\emph{Proof.} \begin{proof}Let $q$ be the dual H\"older exponent of $p$ and consider the map linear map $H : \dot W_0^{k,q}(\Omega) \longrightarrow L^q(\Omega)^{\binom{N + k - 1}{k}}$ defined by the assignment $u \mapsto D^k u$. By definition of the homogeneous norm, $H$ defines a one-to-one isometry and therefore, by the Hahn Banach theorem, $H^* : L^p(\Omega)^{\binom{N + k - 1}{k}}/\ker H^* \to \dot W^{-k,p}(\Omega)$ is well-defined and is a bijective isometry. By construction, $H^*$ is precisely the distributional operator $(-1)^k \diverg^k $ and hence  $$\ker H^* = N_p(\diverg^k,\Omega),$$ as desired.
\end{proof}

The following result will be crucial to construct a linear extension of the solution operator $\Acal^{-1}$ (see Lemma~\ref{lem:existence}), to all of $W^{-k,p}(\Omega;W)$, when $p = 2$ or when $\Omega$ is a sufficiently regular and connected domain. 

\begin{proposition}\label{lem:Sob}
Assume that either $p = 2$ or that $\Omega$ is a connected domain with smooth boundary $\partial \Omega$. Then, there exists a one-to-one linear map $R: \dot W^{-k,p}(\Omega) \longrightarrow L^p(\Omega)^{\binom{N + k - 1}{k}}$ satisfying 
\[
	\sigma = (-1)^k \diverg^k (R\sigma) \quad \text{for all $\sigma \in \dot W^{-k,p}(\Omega)$}.
\]
Moreover, this map is bounded in the sense that 
\[
	\|R\sigma\|_{L^p(\Omega)} \le C(p,k,N,\Omega) \, \|\sigma\|_{W^{-k,p}(\Omega)}.
\]
\end{proposition}
\emph{Proof.} \begin{proof}First, we observe that when $p = 2$, the assertion is a straightforward consequence of Hilbert spaces' theory. Indeed, since every closed subspace of a Hilbert space has an orthogonal complement, the assertion follows directly from the previous proposition. We now address the issue when $\Omega$ is a connected domain with a regular boundary. In this scenario, the classical existence and regularity theory for (powers of) the Laplacian implies the following: for every $g \in L^p(\Omega)^{\binom{N + k - 1}{k}}$, there exists a unique solution $u_g \in W_0^{k,p}(\Omega)$ of the equation
\[
	\Delta^k u \coloneqq \diverg^k(D^k u) = (-1)^k\diverg^k(g) \quad \text{on $\Omega$}.
\]
Moreover, the solution $u_g$ satisfies the $L^p$-estimate
\[
	\|D^k u_g\|_{L^p(\Omega)} \le C(p,k,N,\Omega) \, \|g\|_{L^p(\Omega)},
\]
where the constant depends on the diameter of $\Omega$ and on $\partial \Omega$ (this requires its boundary $\partial \Omega$ to be sufficiently regular). By the previous proposition, we deduce that, for every $\sigma \in \dot W^{-k,p}(\Omega)$, the equation
\[
	\Delta^k u = \sigma \quad \text{on $\Omega$}
\]
has a unique solution $u_\sigma \in W_0^{k,p}(\Omega)$ satisfying
\[
	\|D^k u_\sigma\|_{L^p(\Omega)} \le C(p,k,N,\Omega) \, \|\sigma\|_{\dot W^{-k,p}(\Omega)}.
\]
Since this solution unique for every such $\sigma$, it follows that the map $R : \sigma \mapsto (-1)^k D^k u_\sigma$ is well defined, and defines a one-to-one bounded linear map from $\dot W^{-k,p}(\Omega)$ to $L^p(\Omega)^{\binom{N + k - 1}{k}}$. Moreover, by construction
\[
	(-1)^k\diverg^k R\sigma = \Delta^k u_\sigma = \sigma \quad \text{and} \quad \|R\sigma\|_{L^p(\Omega)} \le C(p,k,N,\Omega) \, \|\sigma\|_{\dot W^{-k,p}(\Omega)}.
\]
This finishes the proof.
\end{proof}

\begin{notation*} Let $r > s$ be arbitrary non-negative integers. We shall henceforth, under the  embedding $W^{-s,p}(\Omega) \embed W^{-r,p}(\Omega)$, consider $W^{-s,p}(\Omega)$ as a subset of $W^{-r,p}(\Omega)$.\end{notation*}

\subsection{The inverse of maximal-operators} The basis of our results consists in exploiting the existence of \emph{true} fundamental solutions for maximal-rank operators. This is addressed in the following result:

\begin{lemma}[Existence and regularity]\label{lem:existence} Assume that $\Acal$ maximal-rank operator of order $k$ on $\R^n$, from $V$ to $W$. 	Then, there exists a linear and bounded solution operator 
	\[
	\Acal^{-1} :L^p(\Omega;W) \longrightarrow W^{k,p}(\Omega;V),
	\] satisfying classical Sobolev regularity estimates. That is, for every  $f \in L^p(\Omega;W)$, it holds
	\[
		\Acal (\Acal^{-1}[f]) = f \quad \text{on $\Omega$}
	\]
	 and 
	\[
		\|\Acal^{-1}[f]\|_{W^{k,p}(\Omega)} \le C(p,\diam\Omega,A)\, \|f\|_{L^p(\Omega)}.
	\]
	
	Moreover, there exists a (possibly non-linear) operator 
	\[
	\tilde S : \dot W^{-k,p}(\Omega;W) \longrightarrow L^p(\Omega;V),
	\] 
	which extends $\Acal^{-1}$ and is such that 
\[
		\Acal (\tilde S \sigma) = \sigma \quad \text{on $\Omega$}\,,
	\]
	 for every $\sigma \in W^{-k,p}(\Omega;W)$. Furthermore, for every
 integer $0 \le r \le k$ and every $\sigma \in \dot W^{-r,p}(\Omega;W)$, it holds 
	\[
		\|\tilde S\sigma\|_{\dot W^{k-r,p}(\Omega)} \le C(p,r,\diam\Omega,A)\, \|\sigma\|_{\dot W^{-r,p}(\Omega)}\,.
	\]
	
	If additionally, $p = 2$ or $\Omega$ is a connected domain with smooth boundary $\partial \Omega$, then the solution operator $\tilde S$ above can be constructed as a linear extension of $\Acal^{-1}$ (which we shall still denote by $\Acal^{-1}$).
	%
%
	\end{lemma}
	
\emph{Proof.} \begin{proof}First, we address the existence and properties of $\Acal^{-1}$ on $L^p(\Omega;W)$. The strategy of the proof is to exploit the regularity properties of the Moore-Penrose pseudo-inverse of the principal symbol of $\Acal$, a technique that dates back to the work of Murat~\cite{murat1981compacite}, Fonseca and M\"uller~\cite{fonseca1999quasi}, Gustafson~\cite{gustafson2011poincare}, and more recently also key for the work of Raita~\cite{raita2019potentials} and Arroyo-Rabasa~\cite{adolfo}.  Let us write $k = n + \ell > 0$, where $\ell$ is an integer, so that $\ell > -n$. We recall from~\cite{simental} that  
	\[
		P(\xi) \coloneqq A(\xi)^\dagger\qquad (\xi \in \Rbf^n - \{0\})
	\]
	defines a $-(n+\ell)$-homogeneous distribution in $C^\infty(\Rbf^n - \{0\};\Hom(W;V))$. Here, for a matrix $M$, we have denoted by $M^\dagger$ its Moore--Penrose pseudo-inverse, which satisfies the fundamental algebraic identity $M M^\dagger = \mathrm{proj}_{\im M}$.
	By assumption, $A(\xi)$ has maximal rank for every $\xi \in \Rbf^n-\{0\}$ and by the discussion in the introduction we may hence assume without loss of generality that $\im A(\xi) = W = W_\Acal$ for all non-zero frequencies. In particular, it follows from the universal property of the quasi-inverse that $A(\xi)P(\xi) = \mathrm{proj}_{\im A(\xi)} = \id_W \in \Hom(W,W)$ for all non-zero $\xi \in \R^n$. Moreover, by~\cite[Theorems~3.2.3,~3.2.4]{Hormander}, we find that $P$ can be extended to a tempered distribution 
	\[
	P^{\frarg} \in \mathscr S'(\R^d;\Hom(W,V))
	\]
satisfying $pP^{\frarg} = (p P)^{\frarg}$. Here, $(\_)^{\frarg}$ denotes the extension for all homogeneous polynomials of degree $s > \ell$. ( \footnote{Here, since $s > \ell$, $(p P)^{\frarg} \in \mathscr D'(\Rbf^n;V\otimes W^*)$ denotes the unique extension of the homogeneous distribution $pP\in \mathscr D'(\Rbf^n-\{0\};V \otimes W^*)$.} ) It follows also that its inverse Fourier transform $K \coloneqq \mathrm (2\pi i)^k\Fcal^{-1} ({P^{\frarg}})$ is a smooth homogeneous map of degree $\ell$ on $\R^n - \{0\}$, locally integrable on $\Rbf^n$ (since $\ell > -n$), of the form 
	\[
		K(x) = |x|^\ell \Psi\bigg(\frac{x}{|x|}\bigg) - Q(x) \log |x|,
	\]  
	where $\Psi : \Sbf^{N-1} \to \Hom(W,V)$ is smooth and $Q$ is a $\Hom(W,V)$-valued polynomial  (this follows from~\cite[Thms.~7.1.16-7.1.18; Eqn.~(7.1.19)]{Hormander}, when applied component-wise to the coordinates of $K$). 
%

	Now, let us write $\tilde f$ to denote the (trivial) extension by $0$ of $f$ to $L^p(\Rbf^n;W)$. Since $\tilde f$ is compactly supported (on $\R^n$) and $K$ is locally integrable, we may define a locally $p$-integrable map $v \coloneqq K \star \tilde f \in \mathscr S'(\R^n;V)$. Note that $K$ is, in fact, a fundamental solution of $\Acal$ in the sense that 
	\[
		\Acal K = \delta_{0_W} \quad \text{in $\R^n$}.
	\]	 
Now, let us fix $R > 0$ and notice that by construction  we get
	\begin{equation}\label{eq:est1}
		\|v\|_{L^p(B_R)} \le C \|\tilde f\|_{L^p} = C\|f\|_{L^p(\Omega)},
	\end{equation}
	for some constant depending solely on $N, p,R$ and $\Acal$. Next, let $\alpha$ be a multi-index of order $|\alpha| = k =N + \ell$  and let $p^\alpha(\xi) = (2\pi \mathrm i)^{|\alpha|} \xi^\alpha$ be the polynomial associated to the Fourier transform of $D^\alpha$. Since the degrees of $A$ and $p^\alpha$ are strictly larger than $\ell$, applying the Fourier transform on $\Acal v$ and $D^\alpha$ gives the distributional identities
	\[
		\Fcal(\Acal v) = A\Fcal K [\Fcal \tilde f] = A P^{\frarg} [\Fcal \tilde f ]= (A P)^{\frarg}[\Fcal \tilde f] = \Fcal \tilde f.
	\]
	Moreover,
	\[
		\Fcal(D^\alpha v) = p^\alpha\Fcal K  [\Fcal\tilde f] = p^\alpha P^{\frarg} [\Fcal\tilde f] = (p^\alpha P)^{\frarg} [\Fcal \tilde f]. 
	\]
	Thus, applying the inverse Fourier transform to the first identity gives $\Acal v = \tilde f$ as distributions over $\Rbf^n$. Furthermore, since $\Omega$ is an open set, we obtain that the linear map
	\[
		\Acal^{-1}[f] \coloneqq  (K \star \tilde f)|_\Omega
	\]
	solves the equation
	\[
		\Acal u = f \qquad \text{in $\mathscr D'(\Omega;W)$}.
	\]
	This proves the first statement.

	Applying the inverse Fourier transform to the second identity and observing that $p^\alpha P$ defines a smooth zero-homogeneous map on $\Rbf^n - \{0\}$ allows us to apply Mihlin's multiplier theorem, which gives the homogeneous norm estimate 
	\begin{equation}\label{eq:est2}
		\|D^\alpha (\Acal^{-1}[f])\|_{L^p(\Omega)} \le \|D^\alpha v\|_{L^p} \le C \|\tilde f\|_{L^p} = C\|f\|_{L^p(\Omega)}\,,
	\end{equation}
	for some $C = C(n,p,\Acal)$. 
	
	To conclude, we set $R \coloneqq \diam \Omega$ and observe that up to a translation $x \mapsto x - x_0$ with $x_0 \in \Omega$, and iteration of the classical Poincar\'e inequalities on $W^{n + \ell,p}(B_R)$ and the estimates~\eqref{eq:est1}-\eqref{eq:est2} render the sought estimate
	\begin{equation}\label{eq:est3}
		\|\Acal^{-1}[f]\|_{W^{k,p}(\Omega)} \le \|v\|_{W^{N + \ell,p}(B_R)} \le C(p,\diam \Omega,A) \|f\|_{L^p(\Omega)}.
	\end{equation}
	

	Here, we have used that both $n$ and $\ell$ depend on $\Acal$. 

Now that we have constructed $\Acal^{-1}$, we may construct $\tilde S$ using the representation of negative Sobolev. Let us fix $r$ as in the assumptions. We distinguish the following cases, which yield the second and third statements of the Lemma, respectively: 
\begin{enumerate}[left=0pt,label=(\alph*)]
\item if $p \neq 2$ or no regularity on $\Omega$ is imposed, we define the \textbf{non-linear} map
\[
	 H_r \sigma \coloneqq \mathrm{arg~min} \set{\|g\|_{L^p(\Omega)} }{g\in L^p(\Omega;W)^{\binom{N + r -1}{r}}, \diverg^r g = \sigma};
\]
\item if $p = 2$ or $\Omega$ is a smooth simply connected regular domain, then we define $H_r = S$, where $S$ is the \textbf{linear} map from Proposition~\ref{lem:Sob} for $k = r$.
\end{enumerate}

Let $\sigma \in W^{-r,p}(\Omega;W)$ be given, so that 
\[
	\sigma = \diverg^r(H_r \sigma) 
\]
and
\[
  \|H_r \sigma\|_{L^p(\Omega)} \le C \|\sigma\|_{W^{-r,p}(\Omega)} \,.
\]
Let us define 
\[
 \tilde S_r \sigma \coloneqq \sum_{|\beta| = r}D^\beta (\Acal^{-1}[(H_r \sigma)_\beta])\,,
\]
and notice that $\tilde S_r$ is linear whenever $H_r$ is (the last assertion of the Lemma follows from this).
Moreover, the estimates we proved for $\Acal^{-1}$ convey the following estimate: 
\begin{equation}\label{eq:pera}
\begin{split}
	\|\tilde S_r\sigma\|_{W^{k-r,p}(\Omega)} & \lesssim  \|D^k(\Acal^{-1}[ H_r\sigma])\|_{L^p(\Omega)} \\  
	& \lesssim   \| H_r \sigma\|_{L^p(\Omega)}
\lesssim \|\sigma\|_{W^{-r,p}(\Omega)}\,.
\end{split}
\end{equation}

Lastly, we observe that $\tilde S_r$ is a solution operator in the sense that 
$$\Acal (\tilde S_r \sigma) =  \sigma \quad \text{for all $\sigma \in \dot W^{-r,p}(\Omega;W)$.}$$ 
This follows from the commutativity of distributional derivatives since (using the double-index summation convention)
\[
	\Acal (\tilde S_r \sigma) = D^\beta A_\alpha D^\alpha (\Acal^{-1}[(H_r \sigma)_\beta]) = D^\beta (H_r \sigma)_\beta = \diverg^r H_r\sigma = \sigma.
\]
In particular,
\begin{equation}\label{eq:manzana}
	\Acal (\tilde S_r \Acal u) = \Acal u \qquad \text{whenever $\Acal u \in W^{-r,p}(\Omega;V)$}.
\end{equation}
Notice that, by construction $\tilde S_r$ extends $\tilde S_m$ whenever $k \ge r \ge  m\ge 0$. Thus, defining $\tilde S \coloneqq \tilde S_k$, 
the second and third statements of the lemma follow directly from~\eqref{eq:pera} and~\eqref{eq:manzana}. This completes the proof.
\end{proof}

\section{Proofs of the main results}

%

\subsection*{Proof of Theorem~\ref{thm:2}}
Define $T : W^{\Acal,p}(\Omega) \subset L^p(\Omega;V) \longrightarrow L^p(\Omega;V)$ as
\[
	Tu \coloneqq u - \Acal^{-1}[\Acal u], \quad D(T) = W^{\Acal,p}(\Omega).
\]	
The fact that $T$ is well-defined follows from the triangle inequality and the fact that $\Acal u \in L^p(\Omega;W)$ for all $u \in W^{\Acal,p}(\Omega)$. Clearly, $T$ is linear in its domain of definition due to the linearity of $\Acal^{-1}$ and $\Acal$.
By construction, the composition $\Acal \circ T$ is the zero map. Moreover, $Tu = u$ for all $u \in N_p(\Acal,\Omega)$. This proves that $R(T) = N_p(\Acal,\Omega)$. Notice that  
\[
T^2 u = Tu - \Acal^{-1}[\Acal Tu] \stackrel{(\Acal \circ T \equiv 0)}= Tu \quad \forall u \in D(T),
\]
which proves that $T$ is an unbounded linear projection. We are left to verify the estimates for $T$ and $\id - T$. If $u \in W^{\Acal,p}(\Omega)$, then the estimates for $\Acal^{-1}$ imply
\[
\|Tu\|_{L^p(\Omega)}  \le \|u\|_{L^p(\Omega)} + \|\Acal^{-1}[\Acal u]\|_{W^{k,p}(\Omega)}  \le \|u\|_{L^p(\Omega)} + C \|\Acal u\|_{L^p(\Omega)}.
\]
Similarly,
\[
\|u - Tu\|_{W^{k,p}(\Omega)}  \le \|\Acal^{-1}[\Acal u]\|_{W^{k,p}(\Omega)} \le  C \|\Acal u\|_{L^p(\Omega)}.
\]
This finishes the proof. 

\subsection*{Proof of Corollary~\ref{cor:H1}} That $W^{\Acal,p}(\Omega) = N_p(\Acal,\Omega) + \Acal^{-1}[L^p(\Omega;W)]$ follows from the fact $N_p(\Acal,\Omega) +  W^{k,p}(\Omega;V) \subset W^{\Acal,p}(\Omega)$ and that every $u \in W^{\Acal,p}(\Omega)$ can be written as 
\[
	u = Tu + (u - Tu) = Tu + \Acal^{-1}[\Acal u].
\]
Moreover, $N_p(\Acal,\Omega) \subset W^{k,p}(\Omega)$ is closed because it coincides with the null-space of the linear map $\Acal: W^{\Acal,p}(\Omega) \to L^p(\Omega;W)$.  Let us verify that $R(\Acal^{-1})$ is closed in $W^{k,p}(\Omega)$. Let $g_j \coloneqq \Acal^{-1}[f_j]\in R(\Acal^{-1})$ and assume that $g_j \to g$ strongly in $W^{k,p}(\Omega)$ so that $f_j = \Acal g_j \to \Acal g$ strongly in $L^p(\Omega;W)$. Using the continuity of $\Acal^{-1}$ we get
\[
	\|\Acal^{-1}[\Acal g] - g_j\|_{W^{k,p}(\Omega)} \lesssim \|\Acal g - f_j\|_{L^p(\Omega;W)} \to 0,
\]
which shows that $g = \Acal^{-1}[\Acal g] \in R(\Acal^{-1})$. Since $g_j$ was an arbitrary convergent sequence, this shows that $R(\Acal^{-1})$ is indeed closed in $W^{k,p}(\Omega)$. We are left to prove that  if
\[
	v + \Acal^{-1}[f] = 0, \quad \text{for some $v \in N_p(\Acal,\Omega), f \in L^p(\Omega;W)$}, 
\]
then $v = \Acal^{-1}[f] = 0$. To see this, we apply $\Acal$ to this identity to get $f = 0$. The conclusion then follows from the identity $v = - \Acal^{-1}[f] = 0$.

\subsection*{Proof of Corollary~\ref{cor:weak}} Let $\tilde S : L^p(\Omega;W) \to L^p(\Omega;V)$ be the (possibly non-linear) extension of $\Acal^{-1}$ constructed in Lemma~\ref{lem:existence}. Then, by the second statement in that lemma and writing $u = \tilde S u - (u - \tilde Su)$, we get
\begin{align*}
	\min_{v \in N_p(\Omega;V)} \|u - v\|_{W^{k-r,p}(\Omega)} & \le \|\tilde S u\|_{W^{k-r,p}(\Omega)}\\
	&  \le C(p,r,\diam\Omega,A) \, \|\Acal u\|_{W^{-r,p}(\Omega)}\,.
\end{align*} 
This finishes the proof. 

\subsection*{Proof of Theorem~\ref{thm:FM} and Corollary~\ref{cor:R}} The proofs are analogous to the proof of Theorem~\ref{thm:2} and Corollary~\ref{cor:H1}, with the exception that, instead of $T$, the proof for these statements appeals to defning the linear map 
\[
\overline Tu \coloneqq u - \Acal^{-1} [\Acal u], \qquad u \in L^p(\Omega;V),
\]
where $\Acal^{-1} = \tilde S$ is (under the assumptions of Theorem~\ref{thm:FM} and Corollary~\ref{cor:R}) the linear inverse map of $\Acal$ from Lemma~\ref{lem:existence} for $r = k$. The boundedness of $\cl T$ (which is the only difference in the proofs) follows from the estimates
\begin{align*}
	\|\cl Tu\|_{L^p(\Omega)} & \le \|u\|_{L^p(\Omega)} + \|\Acal^{-1} u\|_{L^p(\Omega)} \\
	& \lesssim \|u\|_{L^p(\Omega)} + \|\Acal u\|_{\dot W^{-k,p}(\Omega)} \lesssim \|u\|_{L^p(\Omega)}\,,\\
\end{align*}
where in passing to the last inequality, we have used the continuity of the embedding $D^k[L^p(\Omega)] \embed W^{-k,p}(\Omega)$.

\subsection*{Proof of Corollary~\ref{cor:measures}}  Let $1 < q < N/(n-1)$ so that $N < q' < \infty$. If $\Omega$ is sufficiently regular, then Morrey's embedding $W_0^{1,q'}(\Omega) \embed C_0(\Omega)$ holds. Since the image of the embedding is dense, we deduce that 
  $\Mcal_b(\Omega;W) \embed W^{-1,q}(\Omega;W)$, which, together with Thm~\ref{thm:FM}, gives
  \[
  	\| u - \tilde Tu\|_{W^{-1,q}(\Omega)} \lesssim \|\Acal u\|_{W^{-1,q}(\Omega)} \le C(q,N)\, |\Acal u\\|(\Omega).
  \]
  This finishes the proof.

\subsection*{Proof of the weak Korn estimates}

To work out the case when $\Acal$ is not a maximal-rank operator, we must briefly discuss the dichotomy between constant-rank operators and (higher-order elliptic) differential complexes. Recently, Raita~\cite{raita2019potentials} (see also~\cite{simental}) showed that there exists a \emph{surjective} correspondence between higher-order elliptic complexes and constant-rank operators. There, the author proved that if $\Acal$ satisfies the constant rank property, then there exists a differential complex
\[
\xymatrix@1{
&
	{\mathscr D'(\Rbf^n;V)}\ar[r]^-{\Acal} & {\mathscr D'(\Rbf^n;W)}\ar[r]^{\Qcal} & {\mathscr D'(\Rbf^n;X)}}
\]
where $\Qcal$ is a constant-rank operator on $\Rbf^n$, from $W$ to $X$, of order $k_\Qcal$ and such that the symbol complex
\[
\xymatrix@1{
& \;
	{ V}\ar[r]^-{A(\xi)} \; & \;  {W}\ar[r]^{Q(\xi)} \; & \;  {X}}
\]
defines an \emph{elliptic complex}, that is, 
\[
\im A(\xi) = \ker Q(\xi) \qquad \text{for all $\xi \in \Rbf^n - \{0\}$.}
\]
The converse assertion, that symbols of elliptic complexes define operators of constant rank is a trivial consequence of the lower semicontinuity of the rank and the upper semicontinuity of the dimension of the kernel. ( \footnote{This simple argument for the converse statement was shown to me by S. M\"uller, during my Ph.D. studies.} ) 
With these considerations, it is natural to define a generalized Laplace--Beltrami operator on $W$ by setting (cf.~\cite{GR1})
\begin{equation}\label{eq:beltrami2}
\triangle_W \coloneqq [\Acal\Acal^*]^{k_\Qcal} + [\Qcal^*\Qcal]^{k}.
\end{equation} 
\begin{remark}
The associated \emph{annihilator} $\Qcal$ of $\Acal$ is, in general, not uniquely defined, which implies that $\triangle_W$ may also not be uniquely defined.
\end{remark}
The following observation will allow us to make use of the Lemma~\ref{lem:existence} and Theorem~\ref{thm:2} when considering the operator $\Acal = \triangle_W$. 
\begin{proposition}The operator $\triangle_W$ is an elliptic system (as in Definition~\ref{def:es} below). In particular, $\triangle_W$ is a maximal-rank operator.
\end{proposition}
\begin{proof}Let us write $\triangle = \triangle_W$. Since $\triangle$ is a homogeneous operator from $W$ to $W$, we are only left to verify that $\triangle$ is elliptic.	By construction we have that $p(\xi)[w] = \dprb{\triangle(\xi)w,w} = |A(\xi)^*[w]|^{2k_\Qcal} + |Q(\xi)[w]|^{2k}$ for all $\xi \in \Rbf^n$ and $w \in W$. In particular, $p$ vanishes if and only if $A(\xi)^*[w] = 0$ and $Q(\xi)[w] = 0$. A well-known linear algebra property of exact sequences is that they split, that is, $W = \im A(\xi) \oplus \im Q(\xi)^*$ so that $w = A(\xi)[a] \oplus Q(\xi)^*[b]$  and $p(\xi)[w] = 0$ if and only if $[A(\xi)^*A(\xi)]a = 0$ and $[Q(\xi)Q(\xi)^*]b = 0$. Since $\ker(M^*M) = \ker (M)$, the equality with zero may only occur provided that $\xi = 0$ or that $A(\xi)[a] = 0$ and $Q(\xi)^*[b] = 0$. Therefore, by the  homogeneity of $\triangle$, we deduce that there exists a constant $C$ satisfying
	\[
	\forall \xi \in \Rbf^n, \qquad |\triangle(\xi)[w]|  \ge C|\xi|^{2k k_\Qcal} |w| \qquad (w \in W).
	\]
	This proves that $\triangle$ is indeed an elliptic operator from $W$ to $W$, and hence it defines an elliptic system. 
\end{proof}

\subsection*{Proof of Theorem~\ref{thm:Z}} Let us recall that the sequence 
\[
\xymatrix@1{
&
	{\mathscr D'(\Rbf^n;V)}\ar[r]^-{\Acal} & {\mathscr D'(\Rbf^n;W)}\ar[r]^{\Qcal} & {\mathscr D'(\Rbf^n;X)}}
\]
defines a differential complex and therefore $\Acal^* \circ \Qcal^* = 0$ 
as distributional differential operators.
In particular, 
\begin{align}
	\Acal^* \circ \triangle_W  & = \Acal^* \circ (\Acal \circ \Acal^*)^{k_\Qcal} \label{eq:AD}
\end{align}
also as differential operators.

Let us write $\tilde S_W$ to denote the (non-linear) solution operator associated with the operator $\triangle_W$ (constructed in Lemma~\ref{lem:existence}). Here, we are using that $\triangle_W$ is a maximal-rank operator from $W$ to $W$. Define 
\[
	\Pi u \coloneqq u - \Acal^* (\Acal \Acal^*)^{k_\Qcal-1}  \tilde S_W \Acal u, \qquad u \in L^p(\Omega;V).
\]
Using the estimates for $\tilde S_W$, we obtain the estimate (recall that $\triangle_W$ is of order $2kk_\Qcal$) 
\begin{align*}
	\|\Pi u\|_{L^p(\Omega)} &  \le \|u\|_{L^p(\Omega)} + \|\tilde S_W \circ \Acal u\|_{W^{2kk_\Qcal - k,p}(\Omega)} \\
	& \lesssim \|u\|_{L^p(\Omega)} + \|\Acal u\|_{W^{-k,p}(\Omega)} \lesssim \|u\|_{L^p(\Omega)} \,.
\end{align*}
Similarly, for each integer $0 \le r \le k$, 
\begin{equation}
\begin{split}
	\|u - \Pi u\|_{W^{k-r,p}(\Omega)} &  \le \|\tilde S_W \circ \Acal u\|_{W^{2kk_\Qcal - r,p}(\Omega)}  \\
	& \lesssim  \|\Acal u\|_{W^{-r,p}(\Omega)}  \,. \label{eq:W}
\end{split}
\end{equation}
Moreover, by construction it holds
\begin{equation}
\begin{split}
	\Delta_\Acal\Pi u 
	& \;=\; \Delta_\Acal u - \Acal^* \Delta_\Acal^{k_\Qcal}\tilde S_W \Acal u \\
	& \stackrel{\eqref{eq:AD}}= \Delta_\Acal u - \Acal^* \triangle_W\tilde S_W \Acal u = \Delta_\Acal u - \Delta_\Acal u = 0, \label{eq:WW}
\end{split}
\end{equation}
which shows that $\im \Pi \le N_p(\Delta_{\Acal};\Omega)$.
The sought Korn distance estimates then follow from~\eqref{eq:W}-\eqref{eq:WW}.


\section{Examples}\label{sec:examples} In this section, we collect examples of several operators satisfying the maximal-rank property, for which Theorem~\ref{thm:2} applies. 

\begin{notation*}In order to keep the statements as simple as possible, we shall henceforth use the convention that $\|\sigma\|_{L^p(\Omega)} = \infty$ whenever $\sigma \notin L^p(\Omega)$.
\end{notation*}

\subsection{The Divergence} The  row-wise divergence (and in particular the scalar divergence when $N=1$) defined by
\[
\diverg u = \bigg(\sum_{j = 1} \frac{\partial u^i_j}{\partial x_j}\bigg)_{i = 1,\dots,N}, \qquad u:\Rbf^n \to \Rbf^{N \times n},
\]
is a maximal-rank operator on $\R^n$, from $\Rbf^{N \times n}$ to $\Rbf^N$. 
\begin{proof}
	It suffices to show that $\diverg(\xi)$ has rank $N$ for all non-zero $\xi \in \R^n$. Indeed, $\diverg(\xi)[v \otimes \xi] = (v \otimes \xi)\cdot \xi=v$ for every $v \in \Rbf^N$ and all $\xi \in \Rbf^n-\{0\}$. 
\end{proof}
In particular, we have the following Sobolev-distance estimates:

\begin{proposition} There exists a constant $C = C(p,n,N,\diam\Omega)$ such that
	\[
	\inf_{\substack{v \in L^p(\Omega;V),\\\diverg v = 0}}\|u - v\|_{W^{1,p}(\Omega)} \le C\|\!\diverg u\|_{L^p(\Omega)}
	\]
	for all $p$-integrable tensor-fields $u: \Omega \to \Rbf^{M \times N}$ with $p$-integrable divergence.
\end{proposition}


One may also consider the $k$\textsuperscript{th} order divergence operator
\[
\diverg^k U = \sum_{|\alpha| = k} D^\alpha U_\alpha, \qquad U = (U_\alpha)_{|\alpha| = k},
\]
which defines a maximal-rank  $k$\textsuperscript{th} order operator on $\Rbf^n$, from $\Rbf^{\binom{N + k -1}{k}}$ to $\Rbf$, for which we get the following estimate:
\begin{proposition}
	Let $k$ be a positive integer. Then
	\[
	\inf_{\substack{V \in L^p(\Omega;\R^n),\\\diverg^k V = 0}} \|U - V\|_{W^{k,p}(\Omega)} \le C(p,k,n,\diam\Omega)\|\diverg^k U\|_{L^p(\Omega)}
	\]
	for all $U \in L^p(\Omega;\Rbf^{\binom{N + k - 1}{k}})$.
\end{proposition}
\begin{proof}The associated principal symbol is defined for $\xi \in \Rbf^n$ and $a \in \Rbf^{\binom{N + k -1}{k}}$ by
	\[
	\diverg^k(\xi)[a] = \sum_{|\alpha| = k} a_\alpha \xi^\alpha.
	\]
	In this case $\diverg^k(\xi)[\otimes^k \xi] = |\xi|^{2k}$, which shows that $\diverg^k$ is also a maximal-rank operator.
\end{proof}

\subsection{The Laplacian} The Laplacian operator 
\[
\Delta u = \sum_{i = 1}^n \frac{\partial^2 u}{\partial x_i}, \qquad u:\Rbf^n \to \Rbf,
\]
is a second order operator on $\Rbf^n$, from $\Rbf$ to $\Rbf$. Its principal symbol is given by the map $\xi \mapsto |\xi|^2$, so that in particular $\Delta$ has maximal rank $1$. In light of Theorem~\ref{thm:2} we obtain the following Sobolev estimate in terms of the distance to the space of harmonic maps:
\begin{proposition} There exists a constant $C = C(p,n,\diam\Omega)$ such that
	\[
	\inf_{\substack{v \in L^p(\Omega;V),\\ v \; \textnormal{harmonic}}}\|u - v\|_{W^{2,p}(\Omega)} \le C\|\Delta u\|_{L^p(\Omega)}\,,
	\]
	for all $p$-integrable functions $u : \Omega \to \R$.
\end{proposition} 
A similar statement holds for the bi-Laplacian operator
\[
\Delta^2 u = \sum_{i,j = 1}^n \frac{\partial^2 u}{\partial x_j}\frac{\partial^2 u}{\partial x_i}, \qquad u : \Rbf^n \to \Rbf.
\]

\subsection{The Cauchy--Riemann equations}
The operator 
\begin{align*}
Lu \coloneqq \left(\frac{\partial u}{\partial x_1} - \frac{\partial  v}{\partial x_2},\frac{\partial u}{\partial x_2} + \frac{\partial  v}{\partial x_1} \right), \qquad u,v : \Rbf^2 \to \Rbf,
\end{align*}
associated to real-coefficient equations arising from the del-var operator
\[
\partial_{\overline{z}} = \frac 12 (\partial_x + \mathrm i \partial_y), \qquad f(x + \mathrm iy) = u(x,y) + \mathrm i v(x,y),
\]
conforms a first-order system on $\Rbf^2$, from $\Rbf^2$ to $\Rbf^2$, with associated principal symbol tensor
\[
\Lbb(\xi) = \begin{pmatrix}
\xi_1 & -\xi_2 \\
\xi_2 & \xi_1
\end{pmatrix} \qquad (\xi \in \Rbf^2).
\]
Since $\det \Lbb(\xi) = |\xi|^2$ for all $\xi \in \Rbf^2$, it follows that the Cauchy--Riemann equations conform a maximal-rank system in the sense of Theorem~\ref{thm:2}. Since $f \in \Hol(\omega)$ if and only if $\partial_{\cl z} f = 0$ on $\omega$, it follows from Theorem~\ref{thm:2} that the distance estimates for holomorphic functions proven in~\cite{Fuchs} also hold for arbitrary (possibly irregular) domains $\omega \subset \Cbb$. The precise statement, which is a refinement of Theorem~\ref{thm:Fuchs}, is the following:
\begin{proposition}\label{prop:20}Let $\omega \subset \Cbb$ be a bounded and open set. There exists a constant $C = C(p,\diam \omega,L)$ such that
	\[
	\inf_{g \, \in \, (\mathrm{Hol} \cap L^p)(\omega)} \|f - g\|_{W^{1,p}(\omega,\C)} \le C \|\partial_{\cl z} f\|_{L^p(\omega)}\,,
	\] 
	for all $p$-integrable complex-functions $f : \omega \to \Cbb$. Here, $W^{1,p}(\omega;\Cbb)$ is the space of functions $f \in L^p(\omega;\Cbb)$ having first-order weak complex partial derivatives in the same space.
\end{proposition} 

In a similar vein, the equations corresponding to the \emph{conformal matrix} inclusion $\nabla u \in K$, where
\[
K \coloneqq \left\{\begin{pmatrix}
a & b \\-b & a 
\end{pmatrix}: a,b \in \Rbf\right\}, \qquad u : \Rbf^2 \to \Rbf^2\,,
\]
conform a first-order system with maximal rank, which conveys the following estimates for gradient conformal maps:
\begin{proposition} There exists a constant $C = C(p,\diam\Omega,K)$ such that
	\[
	\inf_{\substack{V \in L^p(\Omega;\Rbf^2)\,,\\\nabla V \in K}} \|F - V\|_{W^{1,p}(\Omega)} \le C \|L F\|_{L^p(\Omega)} \,.
	\]
	for all vector-fields $F\in L^p(\Omega;\Rbf^2)$.
\end{proposition}

\subsection{The Laplace--Beltrami operator} Let $\ell \in \{1,\dots,n-1\}$. The exterior derivative $d$ is the first-order differential operator from $\bigwedge^\ell \Rbf^n$ to $\bigwedge^{\ell +1}\Rbf^n$ associated to the symbol
\[
d(\xi) = \xi \wedge v \qquad (\xi\in \Rbf^n, v \in \bigwedge^\ell \Rbf^n).
\]
By duality, the co-differential $\delta$ is the first-order operator from $\bigwedge^{\ell+1} \Rbf^n$ to $\bigwedge^{\ell} \Rbf^n$ associated to the symbol
\[
\delta(\xi) = (-1)^\ell \star (\xi \wedge \star v),
\]
where $\star$ is the Hodge-star operator acting on alternating forms. The Laplace--Beltrami operator is the $2$\textsuperscript{nd} order operator, from $\bigwedge^\ell \Rbf^n$ to $\bigwedge^\ell \Rbf^n$ defined as
\[
\triangle u \coloneqq  \delta du + d\delta u,
\]
for which we have the following estimate:
\begin{proposition}
	There exists a constant $C = C(p,n,\ell,\diam\Omega)$ such that
	\[
	\inf_{\substack{w \,\in\, L^p(\Omega;\bigwedge^\ell \Rbf^n)\,,\\ \triangle w = 0}} \|u - w\|_{W^{2,p}(\Omega)} \le C\|\triangle u\|_{L^2(\Omega)}
	\]
	for all $p$-integrable $\ell$-form fields $u : \Omega \to \bigwedge^\ell \Rbf^n$.
\end{proposition}
\begin{proof}Let us fix $\xi \in \Rbf^n$ a non-zero vector. First, let recall that $\delta(\xi)$ is the adjoint $d(\xi)$ so that $\triangle(\xi)= (d,\delta)^* \circ (d,\delta)$. Since $\triangle(\xi) : \bigwedge^\ell \R^n \to \bigwedge^\ell\R^n$, it suffices to observe that $(d(\xi),\delta(\xi))$ is injective, for then it follows that $\triangle(\xi)$ is onto and one can apply the the results of Theorem~\ref{thm:2}. Indeed, $a\in \ker (d(\xi),\delta(\xi))$ if and only if both $\xi \wedge a$ and $\xi \wedge (\star a)$ are zero, which implies that $a = 0$ as desired.
\end{proof}

\subsection{Elliptic systems} A  $k$\textsuperscript{th} order operator on $\R^n$, from $V$ to $W$, is called \emph{elliptic} provided that its principal symbol is injective for all non-zero frequencies, or equivalently, that
\[
\forall \xi \in \Rbf^n, \qquad |A(\xi)[v]| \ge c|\xi|^k|v|
\]
for some  $c > 0$. 

\begin{definition}[Elliptic system]\label{def:es}An \emph{elliptic system} (not to be confused with an elliptic operator) is an elliptic operator with as many indeterminate coordinates as its number of linearly independent equations, that is, an elliptic operator from $V$ to $W$, with
\[
\dim(V) = \dim(W).
\]
\end{definition}
Note that elliptic systems are maximal-rank operators, which yields
 the following projection estimates for elliptic systems:
\begin{proposition}
	Let $\Acal$ be a $k$\textsuperscript{th} order elliptic system on $\Rbf^n$, from $V$ to $W$. Then,
	\begin{equation}\label{eq:ES}
	\inf_{\substack{v \,\in\, L^p(\Omega;V), \\ \Acal v \,= \,0}} \|u - v\|_{W^{k,p}(\Omega)} \le C(p,\diam\Omega,A)\,\|\Acal u\|_{L^p(\Omega;W)}
	\end{equation}
	for all $u \in L^p(\Omega;V)$.
\end{proposition}

\begin{example}[The deviatoric operator]
	The operator associated with the shear part of the symmetric gradient
	\[
	\varepsilon^D(u) \coloneqq \frac 12(Du + Du^T) - \frac{\diverg u}{n} I_n, \qquad u: \Rbf^n \to \Rbf^n,
	\]
	defines an elliptic operator on $\R^n$, from $\Rbf^n$ to the space of real-valued $n \times n$ symmetric trace-free matrices. For $n = 2$, the space of trace-free symmetric $2 \times 2$ matrices has dimension $2$ and hence $\varepsilon^D$ defines an elliptic system for $n = 2$ (thus, also satisfying the projection estimate~\eqref{eq:ES} {on arbitrary domains}).   For $n \ge 3$, it can be shown that $E_D$ has a finite dimensional distributional null-space and therefore the projection estimate~\eqref{eq:ES} also holds {provided that $\Omega$ is a sufficiently regular domain} (cf.~\cite[Theorem 3.3]{Diening1}).
\end{example}
\begin{example}The generalized Laplacian $\Delta_\Acal$ associated with any elliptic operator $\Acal$ from $V$ to $W$ defines an elliptic system from $V$ into itself. Indeed, $\Delta_\Acal$ is also elliptic since it is $2k$-homogeneous and its principal symbol satisfies
	\[
	A(\xi)^*\circ A(\xi)[a] \cdot a = |A(\xi)[a]|^2 \ge C^2|\xi|^{2k}|a|^{2} \qquad (\xi \in \R^n, a \in V). 
	\]
\end{example}

\begin{remark}[elliptic systems vs. ADN systems] ADN systems (from $V$ to $V$) ---as introduced by Agmon, Douglis, and Nirenberg in~\cite{ADN,ADN2}--- are elliptic systems and, in particular, also maximal-rank elliptic operators. However, 
	there exist elliptic systems that fail to be ADN systems (e.g., the Cauchy-Riemann equations). 
\end{remark}

\subsection{Adjoints of elliptic operators}\label{sec:adjoint}
Given an elliptic operator $\Acal$ from $V$ to $W$, its associated formal adjoint $\Acal^*$ from $W$ to $V$ defines an operator with a maximal-rank principal symbol. Indeed, by definition, the principal symbol of $\Acal^*$ is given (up to a sign) by  the algebraic adjoint $A(\xi)^*$, and, since $A(\xi)$ is injective for all non-zero $\xi \in \Rbf^n$, a standard linear algebra argument implies that
\[
\forall \xi \in \Rbf^n - \{0\}, \qquad \im A(\xi)^* = \{\ker A(\xi)\}^\perp = \{0_V\}^\perp = V.
\]
The converse also holds: the formal adjoint of every maximal-rank operator is elliptic.
In particular, the following estimates hold:
\begin{proposition}
	Let $\Acal$ be a $k$\textsuperscript{th} order elliptic operator from $V$ to $W$. Then,
	\[
	\inf_{\substack{v \,\in\, L^p(\Omega;W), \\ \Acal^* v \,= \,0}}	\|u - v\|_{W^{k,p}(\Omega)} \le C(p,\diam\Omega,A^*) \, \|\Acal^* u\|_{L^p(\Omega)}.
	\]
	for all $u \in L^p(\Omega;W)$.
\end{proposition}

%

%

\subsection{Lie group differential constriants (linear stability)} In the following we write $\Mat(m \times n)$ to denote the space of $m\times n$ matrices with real coefficients. When $m = n$, we simply write $\Mat(n)$.

An important problem in the modern compensated compactness theory it is of utter importance to  understand the regularity, rigidity and variational properties conveyed by \emph{non-linear} pointwise differential inclusions of the form
\begin{equation}\label{eq:nonlinear}
	Du(x) \in G, \qquad x \in \Omega,
\end{equation}
where $u : \R^n \to \R^m$ is a map of least possible regularity  and 
 $G$ is either a particular set of points,  a submanifold of $\Mat(m \times n)$, or a Lie group in $\Mat(n \times n)$. For example, one may ask what type of properties on $G$ render the following Liouville rigidity property: if $u \in W^{1,1}_\loc(\R^n)$ satisfies~\eqref{eq:nonlinear}, then $u$ is automatically smooth and in fact there exists a matrix $R \in G$ such that
\[
	Du(x) = R \qquad \text{for all $x \in  \Omega$}.
\]
More generally, a central question is to study the possible quantitative stability of the constraint. Namely, can we find a constant $C$ such that 
\[
	\inf_{v \in \Kcal} \|Du - R\|_{L^p(\Omega)} \le C \|\!\dist(G,Du)\|_{L^p(\Omega)}, 
\]
where $\Kcal$ is the set of all sufficiently regular maps satisfying~\eqref{eq:nonlinear}. In general, this is a very challenging problem to solve with available methods. When $G$ is a sufficiently regular manifold, a necessary condition for the validity of the quantitative stability  is the validity of the linearized stability estimate: 
\[
	\|Du - Dv\|_{L^p(\Omega)} \le C \|\dist(T_p G,Du)\|_{L^p(\Omega)}, \quad p \in G.
\]
This is particularly interesting when $G$ is a compact matrix Lie group, given that the linear stability can be studied in terms of  the validity  of the estimate
\[
	\inf_{D\varphi \in \mathfrak g} \|Du - D\varphi\|_{L^p(\Omega)}  \le \|\!\dist(\mathfrak g,Du)\|_{L^p(\Omega)},
\]
where $\mathfrak g = T_{\id} G$ is the Lie algebra associated with $G$. Both inequalities are of the form~\eqref{eq:start}. Indeed, the distance $\dist(T_p G,Du)$ can be expressed in terms of the operator $\Acal u = P(Du)$, where $P: \Mrm(m \times n) \to \Mrm(m \times n)$ is the canonical projection on the orthogonal complement of $T_p G$. In a more general setting, there are several  variational models in continuum mechanics concerned with the study of the rigidity for the linearized differential inclusion $D^k u(x) \in \Kcal$ (for a vector-valued map $u$ and a suitable Lie group or a suitable subset of $k$-order tensors $\Kcal$), the study of its associated non-linear stability (at $e \in \Kcal$)
\[
	\inf_{v \in L^p(\Omega), \Acal v = 0}\|D^k(u - v)\|_{L^p(\Omega)} \le C(p,\Omega,\Acal) \|\Acal u\|_{L^p(\Omega)}, 
\]
where $\Acal  u= p_{N_e \Kcal}(D^k u)$ is the linear differential operator defined by projecting $Du$ onto the normal $N_e \Kcal$ of the tangent space $T_e \Kcal$. \\ 

Next, we recall some relevant Lie group constraints in Analysis:

\subsubsection{Korn's inequality in linear elasticity theory}\label{Sec:Korn} In the nonlinear theory of elasticity, 
the elastic properties (of a body $\Omega \subset \R^3$) are unchanged under translations and rotations. In a variational setting, this leads one to consider integral energies $\int_\Omega F(Du)$ that are invariant under the action of 
\[
{\SO}(n) = \set {A \in \Mrm(n)}{AA^T = I, \det (A) = 1},
\] 
the set of orthogonal matrices with determinant $1$.  A simple prototype of this model is to consider powers of the distance function  $F(Q) = \dist(Q,\SO(n))$.  Given that these are highly non-convex functions, the study of their associated variational landscape require a delicate analysis of the interplay between low and large energy configurations. This issue stages the paramount relevance of the \emph{analytical rigidity} properties of $\SO(n)$.  Liouville, who studied more general conformality properties, realized that the constraint $Du(x) \in \SO(n)$ is rather rigid on $C^1$ maps (it necessarily holds $Du \equiv const$).  Re\v setnyak~\cite{Resh} showed the same holds for $u \in W^{1,1}(\Omega)$ and he also showed $H^1$-stability under weak convergence: if $Du_j \toweak Du$ and $\dist(Du_j,\SO(n)) \to 0$ in $L^2(\Omega)$, then $Du_j \to Du$ in $H^1(\Omega)$. Friesecke, James and M\"uller~\cite{FM1} (see also~\cite{FM2}) established $L^p$ quantitative stability of the constraint for all $p \in (1,\infty)$: for all $u \in W^{1,p}(\Omega;\R^n)$ it holds
\[
	\inf_{R \in \SO(n)} \|Du - R\|_{L^p(\Omega)}  \le C(p,\Omega)\|\dist(Du,\SO(n))\|_{L^p(\Omega)}.
\]
Prior to this major achievement, the only way to understand the behavior of low-energy elastic configurations was to linearize the energy (hence the so-called \emph{linearized elasticity theory}).  
Since $\SO(n)$ is a compact Lie group and the space of skew-symmetric matrices $\mathfrak{so}(n) = \set{A \in \Mrm(n)}{A^T = - A}$ is its associated (Lie algebra) tangent space  at the identity, the linearized stability estimate is precisely given by Korn's second inequality
\begin{align*}
	\inf_{R \in \mathfrak{so}(n)} \|Du - R\|_{L^p(\Omega)} & \lesssim \|\dist(Du,\so(n))\|_{L^p(\Omega)} \\
	& =\|\sym(Du)\|_{L^p(\Omega)}\,, 
\end{align*}
where $\sym(Du) = \frac 12(Du + Du^T)$ is the so-called symmetric gradient operator. 

\subsubsection{Stability of infinitesimal conformal maps}A map $u: \Omega \to \R^n$ is called conformal if its differential preserves angles up to a dilation. For sufficiently regular maps, this can be expressed in terms of the differential constraint $Du(x) \in \Krm(n)$ for all $x \in \Omega$, where 
\[
\Krm(n) = \set{\lambda Q}{\lambda > 0, Q \in \Srm\Orm(n)}
\]
is the group of conformal matrices. Re\v setnyak observed that the pointwise constraint is satisfied by M\"obius transformations. Therefore, this constraint conveys no Liouville-type rigidity. The conformal group $\Krm(n)$ has a Lie group structure, and its Lie algebra is $\mathfrak k(n) =  \R I_N \oplus \mathfrak{so}(n)$. 
In particular, the space of \emph{infinitesimal conformal maps} 
\[
\set{u  \in L^1_\loc(\Omega;\R^n)}{Du(x) \in \mathfrak{k}(n) \text{ in the sense of disitrubtions}},
\] 
coincides with the kernel of the deviatoric operator 
\[
	\varepsilon^D u = \dev \sym(Du) = \sym(Du) - \frac{\diverg(u)}{n} I_n.
\]
This is a highly regularizing (elliptic) operator, and therefore, all infinitesimal conformal maps are smooth on $\Omega$. 

\textbf{Dimension $n \ge 3$ (FDN).} The kernel of $\varepsilon^D$ is the finite-dimensional space of $n$-dimensional \emph{conformal killing vectors} 
\[
	\dpr{a,x} x - \frac 12 a |x|^2 + Ax + cx + b, \qquad A \in \Srm\Orm(n), \, a,b \in \R^n, \, c \in \R.
\]
By a well-known result of Smith~\cite{smith2},  the finiteness of the kernel of the deviatoric operator is equivalent (on sufficiently regular domains) to the validity of the coercive estimate
\[
 \|u\|_{W^{1,p}(\Omega)} \lesssim \|u\|_{L^p(\Omega)} +  \|\varepsilon^D u\|_{L^p(\Omega)}.
 \]
In turn, by standard tools of functional analysis, this can be seen to be equivalent to the Sobolev-distance estimate
 \[
	\|u - \Pi v \|_{W^{1,p}(\Omega)} \le C \|\varepsilon^D u\|_{L^p(\Omega)},
 \]
 where $\Pi$ is an arbitrary projection from $L^p(\Omega;\R^n)$ to the space of conformal killing vectors restricted to $\Omega$. 
 
 \textbf{Dimension $n = 2$ (Maximal-rank).} The picture is considerably different because the system $\varepsilon^D$ is equivalent to the Cauchy--Riemann equations, which possess a considerably larger kernel:  this is a system with an infinite dimensional kernel (whence the coercive inequality fails),  given by the holomorphic maps on $\Omega$. Despite the lack of a Liouville-type rigidity, our results show that $\ker \varepsilon^D|_\Omega$ has a topological complement in $W^{1,p}(\Omega;\R^2)$. In fact, we show that  there exists a linear projection $\Pi : D(\varepsilon^D|_\Omega) \subset L^p(\Omega;\R^2) \to \ker \varepsilon^D|_\Omega$ satisfying the Sobolev estimate
\[
 	\|u - \Pi u \|_{W^{1,p}(\Omega)} \le C \|\varepsilon^D u\|_{L^p(\Omega)}.
\]
Remarkably, our results show that this estimate holds on arbitrary (possibly irregular and possibly disconnected) bounded domains of $\R^2$.

\subsubsection{Stability of determinant-one matrix gradient-fields} The special linear group $\SL(n)$ consists of all $n\times n$ matrices with determinant $1$. A sufficiently regular change of variables $u : \R^n \to \R^n$ is volume preserving provided that  $Du(x) \in \mathrm{SL}(n)$. In hindsight, with our previous example, this differential constraint is rather loose. It is simple to see that rigidity for this constraint fails, even for Lipschitz maps. Indeed, $\mathrm{SL}(n)$ contains several rank-one connections, which allows for high-order laminations and convex integration methods.    The Lie algebra associated with $\SL(n)$ is the space $\sll(n)$ of $n \times n$ trace-free matrices. Therefore, the stability 
of the differential constraint $Du(x) \in \sll(n)$ corresponds to the validity of the Korn inequality
\[
	\inf_{\substack{v \in L^p(\Omega;\R^n),\\\diverg u = 0}}\|D(u-v)\|_{L^p(\Omega)} \le C \|\!\diverg u\|_{L^p(\Omega)},
\] 
Our results show that this estimate holds on arbitrary bounded sets of $\R^n$. 

\subsubsection{In relation to the Monge--\`Ampere equation} The constraint  $\det(D^2 u) = 1$
can be written as $D^2 u \in \SL_\sym(n)$.  Since the corresponding Lie algebra is the space of all trace-free symmetric  matrices $\sll_\sym(n)$, the stability of the linearized constraint $D^2u \in \sll_\sym(n)$ corresponds with the estimate 
\[
	\inf_{\substack{v \in L^p(\Omega;\R^n),\\\Delta u = 0}} \|D^2(u - v)\|_{L^p(\Omega)} \le C \|\Delta u\|_{L^p(\Omega)},
\]
which also follows directly from our results.

\subsection*{Acknowledgements} The author was supported by the Fonds de la Recherche Scientifique (FNRS) under Grant No 40005112. I would like to thank Lars Diening for sharing with me relevant bibliography about the background theory pertaining to this work. 

\subsection*{Data Availability Statement}
Data sharing is not applicable to this article as no new data were created or analysed in this study.

\subsection*{Author Declarations}
The author has no conflicts to disclose.

\newpage


\newcommand{\noopsort}[1]{}
\providecommand{\bysame}{\leavevmode\hbox to3em{\hrulefill}\thinspace}
\providecommand{\MR}{\relax\ifhmode\unskip\space\fi MR }
\providecommand{\MRhref}[2]{%
  \href{http://www.ams.org/mathscinet-getitem?mr=#1}{#2}
}
\providecommand{\href}[2]{#2}

\enlargethispage*{1cm}
\medskip
\medskip
\medskip
\medskip

\vbox{\parskip0pt\prevdepth-1000pt \multiply \baselineskip by 5 \divide \baselineskip by 6 \parskip0pt
Adolfo Arroyo-Rabasa\\
Institut de Recherche en Mathématique et Physique (IRMP)  \\
Université Catholique de Louvain \\
1348 Louvain-la-Neuve, Belgium}

e-mail: arroyo.rabasa@uclouvain.be

\end{document}